\documentclass[10pt,letterpaper]{amsart}

\let\oldmarginpar\marginpar
\renewcommand\marginpar[1]{\-\oldmarginpar[\raggedleft\footnotesize #1]%
{\raggedright\footnotesize #1}}

\usepackage{amssymb}
\usepackage[all]{xy}
\usepackage{graphicx}
\usepackage[mathcal]{euscript}
\usepackage{verbatim}

\newcommand{\cD}{\mathcal{D}}

\newcommand{\cF}{\mathcal{F}}

\newcommand{\cM}{\mathcal{M}}
\newcommand{\cO}{\mathcal{O}}

\newcommand{\bpi}{\bar{\pi}}

\newcommand{\tU}{{\tilde U}}
\newcommand{\tX}{{\tilde X}}

\newcommand{\id}{\mathrm{id}}
\newcommand{\hpi}{\hat{\pi}}
\newcommand{\br}{\bar{\rho}}

\newcommand{\Z}{\mathbb{Z}}

\newcommand{\z}{\zeta}

\newcommand{\Coh}{\mathfrak{Coh}}

\newcommand{\IC}{\mathrm{IC}}
\newcommand{\cIC}{\mathcal{IC}}

\newcommand{\fn}{\mathfrak{n}}
\newcommand{\fm}{\mathfrak{m}}

\DeclareMathOperator{\bSpec}{\mathbf{Spec}}
\DeclareMathOperator{\Spec}{Spec}
\DeclareMathOperator{\depth}{depth}
\DeclareMathOperator{\codim}{codim}
\DeclareMathOperator{\ccodim}{c-codim}

\DeclareMathOperator{\Supp}{Supp}

\numberwithin{equation}{section}

   \newtheoremstyle{example}{\topsep}{\topsep}%
     {}
     {}
     {\bfseries}
     {}
     {\newline}
     {\thmname{#1}\thmnumber{ #2}\thmnote{ #3}}

\newtheorem{theorem}{Theorem}[section]
\newtheorem{proposition}[theorem]{Proposition}
\newtheorem{lemma}[theorem]{Lemma}
\newtheorem{corollary}[theorem]{Corollary}

\DeclareMathOperator{\Ga}{\Gamma}

\DeclareMathOperator{\pD}{{}^p\cD}
\DeclareMathOperator{\tle}{\tau^p_{\le 0}}
\DeclareMathOperator{\tge}{\tau^p_{\ge 0}}
\DeclareMathOperator{\M}{\mathcal{M}}
\DeclareMathOperator{\Mp}{\mathcal{M}^p_\pm}

\newcommand{\hr}{\hat{\rho}}
\newcommand{\hz}{\hat{\z}}

\newcommand{\g}{\gamma}

\newcommand{\Osh}{\mathcal{O}}

\newcommand{\F}{\mathcal{F}}

\newcommand{\hX}{\hat{X}}

\theoremstyle{definition}
\newtheorem{definition}[theorem]{Definition}

\theoremstyle{remark}
\newtheorem{rmk}[theorem]{Remark}

\theoremstyle{example}
\newtheorem{example}{Example}[subsection]

\title[Generalized Serre conditions]{Generalized Serre
  conditions and perverse coherent sheaves} 
\author{Christopher L.~Bremer}
\address{Department of Mathematics\\
  Louisiana State University\\
  Baton Rouge, LA 70803} 
\email{cbremer@math.lsu.edu} 
\author{Daniel
  S.~Sage} 
\email{sage@math.lsu.edu} 
\thanks{The research of the
  second author was partially supported by NSF grant~DMS-0606300 and
 NSA grant~H98230-09-1-0059.}  

\begin{document}
\begin{abstract}

  In algebraic geometry, one often encounters the following problem:
  given a scheme $X$, find a proper birational morphism $Y \to X$
  where the geometry of $Y$ is ``nicer'' than that of $X$.  One
  version of this problem, first studied by Faltings, requires $Y$ to
  be Cohen-Macaulay; in this case $Y \to X$ is called a
  Macaulayfication of $X$.  In another variant, one requires $Y$ to
  satisfy the Serre condition $S_r$.  In this paper, the authors
  introduce generalized Serre conditions--these are local cohomology
  conditions which include $S_r$ and the Cohen-Macaulay condition as
  special cases.  To any generalized Serre condition $S_\rho$, there
  exists an associated perverse $t$-structure on the derived category
  of coherent sheaves on a suitable scheme $X$.  Under appropriate
  hypotheses, the authors characterize those schemes for which a
  canonical finite $S_\rho$-ification exists in terms of the
  intermediate extension functor for the associated perversity.
  Similar results, including a universal property, are obtained for a
  more general morphism extension problem called $S_\rho$-extension.

\end{abstract}

\maketitle
\section{Introduction}

In algebraic geometry, one often encounters the following problem:
given a scheme $X$, find a proper birational morphism $Y \to X$ where
the geometry of $Y$ is ``nicer'' than that of $X$.  The strongest
version of this problem is the resolution of singularities.  On the
other hand, there are many weaker variations expressed in terms of
local cohomology.  For example, one might require $Y$ to satisfy
Serre's condition $S_2$.  In another version, introduced by
Faltings~\cite{faltings}, $Y$ is required to be Cohen-Macaulay; $Y\to
X$ is then called a Macaulayfication of $X$. Kawasaki has shown that
Macaulayfications exist for a broad class of schemes~\cite{kawasaki},
and they can often be constructed in contexts where desingularizations
do not exist.

In general, Macaulayfications are not canonical.  However, it was shown in
\cite{AS2} that there may exist a finite Macaulayfication, restricting
to an isomorphism over the Cohen-Macaulay locus, that satisfies an
appropriate universal property.  This is a special case of a more
general morphism extension problem.  

Consider the following diagram,
where $U$ is an open dense subscheme 
of $X$ and $\z_1$ is a finite
dominant morphism.
\[
\xymatrix{
*+{\;\tU\;} \ar@{^{(}.>}@<-.5ex>[r]\ar[d]_{\z_1} & \tX \ar@{.>}[d]^{\z} \\ 
*+{\;U\;} \ar@{^{(}->}[r]^{j} & X
}
\]

In \cite{AS2}, Achar and Sage investigated the problem of constructing
an ``$S_2$-extension'' of $(X,\z_1)$: this is a scheme $\tX$ together
with a finite morphism $\z:\tX\to X$ such that $\tX$ contains $\tU$ as
an open subscheme, $\z$ extends $\z_1$, $\tX$ is $S_2$ off of $\tU$,
and $(\tX,\z)$ satisfies an appropriate universal property.  If a pair
satisfies all conditions except for the universal property, it is
called a weak $S_2$-extension.  They applied the theory of perverse
coherent sheaves to show that the $S_2$-extension exists under
suitable hypotheses (for example, the ``componentwise codimension'' of
the complement of $U$ must be at least $2$) .  When $\z_1$ is the
identity, $S_2$-extension gives a canonical $S_2$-ification, which
restricts to an isomorphism over $U$.  Achar and Sage used similar
techniques to show that a canonical finite Macaulayfication exists
when a certain perverse coherent sheaf is defined and a sheaf (i.e.,
concentrated in degree $0$).  Moreover, in this case, the finite
Macaulayfication coincides with the $S_2$-ification.  This last fact
was first observed in a local ring context by
Schenzel~\cite{schenzel}.

In this paper, we strengthen and generalize these results.  A summary
of the theory of perverse coherent sheaves appears in
Section~\ref{section2}.  In Section~\ref{section3}, we introduce
``generalized Serre conditions''--these are local cohomology
conditions which include the Serre conditions $S_r$ and the
Cohen-Macaulay condition as special cases.  A generalized Serre
condition is defined in terms of a function $\rho:\Z_{\ge 0}\to
\Z_{\ge 0}$ which has slope at most one and satisfies $\rho(0)=0$.
These conditions have a close relationship to perversities in the
theory of coherent sheaves, and we show how to associate an
$S_\rho$-perversity to the pair $(X,U)$ in Definition~\ref{prho}.  We
also provide an example of an $S_\rho$ variety that is ``strictly
$S_\rho$ through codimension $n$''.  We then investigate the
$S_\rho$-extension problem in Section~\ref{section4} and show that
under appropriate assumptions, an $S_\rho$-extension exists if and
only if a certain intermediate extension sheaf (with respect to the
$S_\rho$-perversity) exists and is a sheaf (Theorem~\ref{extension}).
If it exists, it is the only weak $S_\rho$-extension; moreover, it
coincides with the $S_2$-extension.  This result is applied to the
finite $S_\rho$-ification problem in Theorem~\ref{srification}.  We
obtain similar results when the codimension condition is relaxed,
although here stronger hypotheses are required.

\section{Serre conditions and perverse coherent
  sheaves}\label{section2}

Throughout the paper, $X$ will be a semi-separated scheme of finite
type over a Noetherian base scheme $S$ admitting a dualizing complex.
We will further assume that $X$ is equidimensional.  For any $x \in
X$, we will write $\bar{x} \subset X$ for the subscheme corresponding
to the closure of $\{ x\}$.  We will write $\codim x$ for
$\codim\bar{x}$.  If $Z$ is a closed subscheme of $X$, we will let
$\ccodim(Z)$ denote the ``componentwise codimension'' (or
``c-codimension'') of $Z$ in $X$; if the $X_i$'s are the irreducible
components of $X$, then
\begin{equation*}\ccodim(Z)=\min_{Z\cap X_i\ne\varnothing}\codim_{X_i}Z\cap
  X_i.
\end{equation*}

Let $\Coh(X)$ be the category of coherent sheaves on $X$, and let
$\cD(X)$ be the bounded derived category of $\Coh(X)$.  Suppose that
$i_x : \{x\} \to X$ is the inclusion of a point.  If $\F \in \Coh(X)$,
we define $i_x^* (\F)$ to be the stalk of $\F$ at $x$.  Since $i_x^*$
is exact, it induces an exact functor from $\cD(X)$ to the bounded
derived category of $\Osh_{X, x}$-modules (written $\cD(\Osh_{X,x})$).
We define $\Ga_x (\F)$ to be subsheaf of $i_x^* (\F)$ consisting of
sections with support on $\overline{\{x\}}$.  This is a left exact
functor, and we write $i_x^! : \cD(X) \to \cD(\Osh_{X, x})$ for the
corresponding derived functor.

Recall that the \emph{depth} of a coherent sheaf $\cF$ at $y$ is
defined to be $r$ if $H^r(i^!_x\cF)$ is the first non-vanishing local
cohomology sheaf; in other words, $H^r (i^!_x\cF)\ne 0$ and $H^k
(i^!_x\cF)=0$ for $k < r$.  We will denote the depth of $\cF$ at $y$ by
$\depth_y (\cF)$. 

\begin{definition}
Let $r$ be a positive integer.
We say that a coherent sheaf $\cF \in \Coh(X)$ is $S_r$ if $\depth_x (\cF) = \dim_x(\cF)$
for all $x \in X$ satisfying $\dim_x(\cF) \le r$.  The sheaf $\cF$ is Cohen-Macaulay
if $\depth_x(\cF) = \dim_x(\cF)$ for all $x \in X$.
\end{definition}

Now, we recall the Deligne-Bezrukavnikov theory of perverse coherent
sheaves and its extension by Achar and Sage~\cite{AB, AS2}. 
\begin{definition}
A \emph{perversity} is a  function $p : X \to \Z$ satisfying
\begin{equation*}
\begin{aligned}
p(y) & \ge p(x) & \quad &\text{and} \\
\codim(y) - p(y) & \ge \codim(x) - p(x) & \qquad &\text{whenever $\codim (y) \ge \codim (x)$.}
\end{aligned}
\end{equation*}
(In particular, $p(x)$ only depends on $\codim(x)$.)  Given a
perversity $p$, we define the dual perversity $\bar{p}$ by $\bar{p}
(x) = \codim(\bar{x}) - p(x)$.
\end{definition}

By \cite[Theorem 3.10]{AB}, a perversity determines two full subcategories,
$\pD(X)^{\le 0}$ and $\pD(X)^{\ge 0}$, such that $(\pD(X)^{\le 0}, \pD(X)^{\ge 0})$
is a $t$-structure on $\cD(X)$.
Specifically, 
\begin{align*}
\pD(X)^{\le 0} &= \{ \F \in \cD(X) \mid \forall x \in X, H^k (i^*_x
(\cF)) = \{0\}  \text{ whenever }
k > p(x) \} \\
\pD(X)^{\ge 0} &= \{ \F \in \cD(X) \mid \forall x \in X, H^k (i^!_x
(\cF)) = \{0\} \text{ whenever } k < p(x) \}.
\end{align*}
We call this $t$-structure the \emph{perverse} $t$-structure with
respect to the perversity $p$.  There are associated truncation
functors $\tle : \cD(X) \to \pD(X)^{\le 0}$ and $\tge : \cD(X) \to
\pD(X)^{\ge 0}$.  The heart of the $t$ structure, denoted $\M^p (X)$,
is the category of \emph{perverse coherent sheaves} with respect to
$p$.   For example, if $p$ is the trivial perversity defined by $p(x) =
0$ for all $x \in X$, then one easily sees that $\M^p (X)$ is the
usual category of coherent sheaves on $X$. 

One of the powerful tools in the theory of perverse coherent sheaves
is the intermediate extension functor.  Suppose that $U \subset X$ is
an open dense subscheme of $X$, and let $Z = X\setminus U$.  In
certain cases, there exists an intermediate extension functor $\cIC^p$
which defines an equivalence between a subcategory of $\M^p (U)$ and a
subcategory of $\M^p (X)$.  The definition of $\cIC^p$ requires the
construction of two new perversities associated to $p$, denoted $p^+$
and $p^-$, which depend on the pair $(X,U)$ (although this dependence
will be suppressed in the notation).  In order to ensure that the
domain of $\cIC^p$ is non-empty, we will eventually need to impose
conditions on the perversity $p$ as well as on $\ccodim (Z)$.

\begin{definition} Fix a pair $(X,U)$ as above.  Let $p$ be a
  perversity on $X$, and let $z$ be any point in $X$ such that
  $\codim(z)=\ccodim(Z)$.  We define
\begin{equation*}
p^- (x) = \left\{ 
\begin{array}{ll}
p(x) - 1 & \text{if} \quad p(x) \ge p(z), \\
p(x) & \text{if} \quad p(x) < p(z),
\end{array}
\right.
\end{equation*}
and 
\begin{equation*}
p^+ (x) = \left\{ 
\begin{array}{ll}
p(x) + 1 & \text{if} \quad \codim (x) - p(x) \ge \codim (z) - p(z), \\
p(x) & \text{if} \quad \codim (x) - p(x) < \codim (z) - p(z).
\end{array}
\right.
\end{equation*}
Additionally, we define $\Mp(U)$ and $\Mp(X)$
to be the full subcategories
\begin{equation*}
\begin{aligned}
\Mp(U) &= {}^{p^-} \cD(U)^{\le 0} \cap {}^{p^+} \cD(U)^{\ge 0} \subset \M^p (U), \\
\Mp(X) &= {}^{p^-} \cD(X)^{\le 0} \cap {}^{p^+} \cD(X)^{\ge 0} \subset \M^p (X).
\end{aligned}
\end{equation*}
\end{definition}

\begin{proposition}[{\cite[Proposition 2.3]{AS2}}]
\label{equiv}
Let $j : U \to X$ be the inclusion map.  Then,
$j^* :  \Mp (X) \to \Mp (U)$ is an equivalence of categories.
\end{proposition}

\begin{definition}\label{IC}
The intermediate extension functor 
$\cIC^p(X, \cdot) : \Mp(U) \to \Mp (X)$ is defined as the inverse
equivalence to that of Proposition~\ref{equiv}.  
\end{definition}

 \begin{rmk}\label{codim}
   By \cite[Remark 2.7]{AS2}, if $U$ is irreducible, then the category
   $\Mp (U)$ reduces to the zero object whenever $\codim(Z) \le
   \codim(U) +1$.
 \end{rmk}

From now on, unless otherwise mentioned, we will assume that $U$ is an
open dense subset of $X$ and that $\ccodim Z\ge 2$.  These conditions
are needed for many of the results we use from \cite{AS2}.  Moreover,
we will primarily consider \emph{standard} perversities on $X$.

\begin{definition}\label{standard}
We say that a perversity $p$ on $X$ is \emph{standard} if
\begin{equation*}
p(x) = p^{-} (x)  = p^{+}(x) = 0 \quad \text{if} \quad \codim (x) = 0.
\end{equation*}
\end{definition}
There are unique maximal and minimal standard perversities on $X$
defined by:
\begin{equation*}
\begin{aligned}
s (x) = \left\{ 
\begin{array}{ll}
0, \\ 1,
\end{array}
\right.
& \quad
c (x) = \left\{ 
\begin{array}{ll}
\codim(x) &\text{if} \, \codim(x) < \ccodim (Z), \\ 
\codim(x) -1 &\text{if} \, \codim(x) \ge \ccodim (Z).
\end{array}
\right.
\end{aligned}
\end{equation*}
\begin{lemma}[{\cite[Lemma 3.3]{AS2}}]
Every standard perversity $p$ satisfies $s(x) \le p(x) \le c(x)$ for all $x \in X$.
\end{lemma}

\begin{rmk}\label{neg}
If $p$ is any standard perversity, then $p^-\ge 0$.  It follows that
any coherent sheaf $\cF$ is contained in ${}^{p^-}\cD(X)^{\le 0}$.
\end{rmk}

\begin{rmk}\label{jstar} The complex $\cIC^s (X, \cF)$ is automatically a sheaf if
  it is defined.  Indeed, it is just $j_* \cF$ by \cite[Proposition
  3.7]{AS2}.
\end{rmk}
\section{Generalized Serre Conditions}\label{section3}

\subsection{The conditions $S_\rho$}\label{defns}
In this section, we introduce a class of local cohomology conditions
which generalize Serre's conditions $S_r$ and the Cohen-Macaulay
condition.  We also show their connection to perversity functions.

Let $W'$ be the set of weakly increasing functions $\rho:\Z_{\ge 0}\to
\Z_{\ge 0}$ such that $r(0)=0$ and $\rho(k+1)-\rho(k)\le 1$ for all
$k$.  Note that $W'$ is a lattice with respect to the usual partial
order with the identity (resp. the zero function) as the maximum
(resp. minimum).  If we set $\rho_r(k)=\min(k,r)$, then $\{\rho_r\}$
is an increasing sequence whose supremum is $\id$. Let $W=\{\rho\in
W'\mid \rho\ge\rho_2\}$.

\begin{definition} Given $\rho\in W'$, we say that $\cF\in\Coh(X)$ is
  $S_\rho$ at $x$ if $H^k(i^!_x(\cF))=\{0\}$ for
  $k<\rho(\dim_x(\cF))$; $\cF$ is $S_\rho$ if it is $S_\rho$ at $x$
  for all $x\in X$.
\end{definition}

Note this says that $\cF$ is $S_\rho$ at $x$ if and only if
$\depth_x(\cF)\ge \rho(\dim_x(\cF))$ for all $x\in X$.  In particular,
$S_{\rho_r}$ is the usual condition $S_r$ while a sheaf is $S_{\id}$
if and only if it is Cohen-Macaulay.  Elements of $W$ correspond to
local cohomology conditions which are at least as strong as $S_2$.

To see the relationship between these conditions and perversities, we
need a numerical version of perversity functions.

\begin{definition} For $n\ge 2$, a \emph{standard numerical perversity
    of level n} is a function $\pi:\Z_{\ge 0}\to \Z_{\ge 0}$
  such that $\pi(0)=0$, $0<\pi(n)<n$, and both $\pi$ and its dual
  $\bpi\overset{\text{def}}{=}\id-\pi$ are nondecreasing.  The set of
  all such functions is denoted by $P_n$; these satisfy $P_n\subset
  P_m$ if $n\le m$.  A standard numerical perversity is an element of
  $P=\cup_{n\ge 2} P_n$.
\end{definition}
Note that $0\le \pi(k+1)-\pi(k)\le 1$ for any $\pi\in P$.

Given $(X,U)$, any element $\pi\in P_{\ccodim Z}$ induces a standard
perversity $p_\pi=\pi\circ\codim:X\to\Z_{\ge 0}$.  Conversely, any
perversity comes from a (non-unique) element of
$P_{\ccodim Z}$.

For $\pi\in P_n$, we set \begin{equation*}\pi^+_n(k) =
\begin{cases}
p(k) + 1 & \text{if $k - p(k) \ge n-p(n)$,} \\
p(k) & \text{if $k-p(k) < n-p(n)$.}
\end{cases}
\end{equation*}
We will suppress $n$ from the notation when it is unambiguous.  In
particular, if we are considering a pair $(X,U)$, then we will always
take $n=\ccodim(Z)$.  With this convention, we see that
$p_{\pi^+}=(p_\pi)^+$.

Given $\rho:\Z_{\ge 0}\to \Z_{\ge 0}$ such that $r(0)=0$ and $n\ge 2$,
let $P_n(\rho)$ be the set of standard numerical perversities $\pi\in P_n$ such
that $\pi^+(k)\le\rho(k)$ for all $x$ with equality if $k\ge n$.  Given
$\rho\in W$ and $n\ge 2$, let $m$ be the largest index such that
$\rho(m)<\rho(n)$.  We then define $\hpi_{\rho,n},\pi_{\rho,n}\in P_n$ via
\begin{align*}
\hpi_{\rho,n}(k) &=
\begin{cases} \rho(k),& \text{ if $k\le m$},\\ \rho(k)-1,& \text{ if $k\ge m+1$,}
\end{cases}
&\pi_{\rho,n}(k)=
\begin{cases} \max(k-(n-\rho(n)+1),0),& \text{ if $k<n$,}\\
\rho(k)-1, & \text{if $k\ge n$.}
\end{cases}
\end{align*}
It is easy to check that these are indeed standard numerical
perversities in $P_n(\rho)$.  We also define a function $\phi:P_2\to W$:
\begin{equation*} \phi(\pi)(k)=\begin{cases} k, &\text{if $k\le 1$,}\\
\pi(k)+1, &\text{if $k\ge 2$.}
\end{cases}
\end{equation*}

\begin{proposition}\label{p3.3}  \mbox{}  \begin{enumerate}  \item Let $\rho:\Z_{\ge 0}\to \Z_{\ge 0}$ be a function such
    that $\rho(k)=k$ for $k\le 1$.  The set $P_2(\rho)$ is nonempty if
    and only if $\rho\in W$.
\item 
For any $\rho\in W$ and $n\ge 2$, $\pi_{\rho,n}$ (resp. $\hpi_{\rho,n}$) is
the unique minimum (resp. maximum) element of $P_n(\rho)$.
\item There exists $\pi\in P_n(\rho)$ such that $\pi^+_n=\rho$ if and only
  if $\rho(n-1)<\rho(n)$.
\item The function $\phi:P_2\to W$ is surjective and two-to-one with
  $\phi^{-1}(\rho)=P_2(\rho)=\{\pi_{\rho,2},\hpi_{\rho,2}\}$.  Moreover, the
  duality map on numerical perversities induces a duality map
  $\rho \mapsto \br$ on $W$ such that $\phi(\pi_{\rho,2})=\hpi_{\br,2}$ and $\phi(\hpi_{\rho,2})=\pi_{\br,2}$.
\end{enumerate}
\end{proposition}

\begin{proof} If $\pi\in P_2(\rho)$, then $\pi(k)=\rho(k)-1$ for $k\ge 2$.
  Also, $\pi(2)=1$, so $\rho(2)=2$.  This implies that $\rho$ is
  nondecreasing and $\rho\ge \rho_2$.  Moreover, since $\pi(k+1)-\pi(k)\le
  1$, the same holds for $\rho$.  Thus, $\rho\in W$.  The converse holds,
  since $\pi_{\rho,2}\in P_2(\rho)$ for $\rho\in W$.

  For the second part, note that if $\pi\in P_n(\rho)$, then
  $\pi(n)=\rho(n)-1$.  If $\pi(k)<k-(n-\rho(n)+1)$ for $n-\rho(n)+1\le k <
  n$, then $\pi(n)<\rho(n)-1$, a contradiction, so $\pi\ge \pi_{\rho,n}$.
  On the other hand, $\pi\le \rho$, so if $\pi\nleq \hpi_{\rho,n}$, there
  exists $k$ such that $m<k<n$ such that $\pi(k)=\rho(k)=\rho(n)$.
  Since $\pi$ is nondecreasing, we obtain $\pi(n)\ge r(n)$, which is
  again a contradiction.  

  Note that $\hpi_{\rho,n}^+(k)=\rho(k)$ except for $m<k<n$, so
  $\hpi_{\rho,n}^+(k)=\rho$ precisely when $\rho(n-1)<\rho(n)$.  Since $\pi\le
  \pi'$ implies $\pi^+\le \pi'^+$, the third statement follows.

  For $\pi\in P_2$, it is immediate that $\pi\in P_2(\rho)$ if and only
  if $\phi(\pi)=\rho$. In addition, $\pi\in P_2(\rho)$ is uniquely
  determined by $\pi(1)\in\{0,1\}$, so $\pi_{\rho,2}$ and $\hpi_{\rho,2}$
  are the only elements of $\phi^{-1}(\rho)=P_2(\rho)$.  The final
  statement about  duality is obvious.
\end{proof}

\begin{definition}\label{Srdef} Given $\rho\in W$, we say that $S_{\br}$ is the generalized
  Serre condition dual to $S_\rho$.
\end{definition}
\begin{example} The Serre condition dual to $S_r$ is given by
  $S_{\br_r}$, where 
\begin{equation*} \br_r(k)=\begin{cases} k,& \text{if $k\le 1$;}\\
2,& \text{if $2\le k\le r$;}\\
k-(r-2), &\text{if $k>r$.}
\end{cases}
\end{equation*}
In particular, $S_2$ and Cohen-Macaulay condition are dual to each
other.
\end{example}

\begin{proposition}\label{ICdefined} Let $\pi\in P_{\ccodim Z}(\rho)$, and suppose that
  $\cF\in\Coh(U)$.
\begin{enumerate}
\item The complex $\cIC^{p_\pi}(X,\cF)$ is defined if and only if
  $\cF\in {}^{p^+_\pi}\cD(U)^{\ge 0}$.  In particular, this is true if
  $\Supp(\cF)=U$ and $\cF$ is $S_\rho$.
\item If $\cIC^{p_\pi}(X,\cF)$ is a sheaf, it is $S_\rho$ at all points not in $U$.
\end{enumerate}
\end{proposition}  

\begin{proof} Since $\cF$ is a sheaf and $p_\pi$ is standard, $\cF$
  lies in ${}^{p^-_\pi}\cD(U)^{\le 0}$ by Remark~\ref{neg}.  Hence,
  $\cIC^{p_\pi}(X,\cF)$ is defined if and only if $\cF\in
  {}^{p^+_\pi}\cD(U)^{\ge 0}$.  If $\Supp(\cF)=U$, then
  $\dim_u\cF=\codim u$ for all $u\in U$, so $\cF\in
  {}^{p^+_\pi}\cD(U)^{\ge 0}$.  Finally, if $x\in Z$, then
  $H^k(i^!_x(\cIC^{p_\pi}(X,\cF))=0$ for all
  $k<p^+_\pi(x)=\rho(\codim(x))$.  If $\cIC^{p_\pi}(X,\cF)$ is a
  sheaf, then it is $S_\rho$ at $x$ because the dimension of a
  coherent sheaf at $x$ is at most $\codim(x)$.
\end{proof}

\begin{definition}\label{prho} The $S_\rho$-perversity on $(X,U)$ is the standard
  perversity $p_\rho(x)=\pi_{\rho,\ccodim Z}(\codim (x))$.
\end{definition}
Note that this perversity gives the least restrictive conditions on
$\cF\in\Coh(U)$ guaranteeing that $\cIC(X,\cF)$ is $S_\rho$ for all
points in $Z$ (if it is a sheaf).  In our previous notation, $s$ is
the $S_2$-perversity.  However, $c$ is not the $S_\id$ (i.e., the
Cohen-Macaulay) perversity, since it corresponds to $\hpi_{\id,\ccodim
  Z}$.

We end this section with an explicit description of what it means for
a sheaf to be in ${}^{p^+_\rho}\cD(X)^{\ge 0}$.
\begin{proposition} Let $n=\ccodim Z$.  A sheaf $\cF$ is in
  ${}^{p^+_\rho}\cD(X)^{\ge 0}$ if and only if
\begin{equation}\label{prhoformula} \depth_x\cF\ge\begin{cases} \rho(\codim (x)),&
    \text{if $\codim (x)\ge n$;}\\
\rho(n)-(n-\codim(x)),& \text{if $n-\rho(n)+1\le \codim(x)<n$;}\\
0, &\text{if $\codim(x)\le n-\rho(n)$.}
\end{cases}
\end{equation}
\end{proposition}
The proof is a simple calculation using the definition of $p^+_\rho$.
Note that the first two cases correspond to points for which
$p_\rho^+(x)=p_\rho^+(x)+1$.

\subsection{Example}

In this section, we will construct an $S_\rho$ scheme $X$ that is
``strictly $S_\rho$ through codimension $n$''.  This example is
adapted from \cite[Section 4]{Gr}.

Let $K$ be an algebraically closed field of characteristic zero, and
let $K[X_1, \ldots, X_{a}, Z]$ be a polynomial ring in $a+1$
variables.  We define a graded ring $T^a$ by $T^a = K[X_1, \ldots,
X_{a}, Z] / (f)$, where $(f)$ is the ideal generated by $f = Z^{a+1} -
\sum_{i = 1}^{a} X^{a+1}_i$.  Let $A^b$ be the polynomial ring
$K[Y_0, Y_1, \ldots, Y_b]$.  We let $X_{a,b}=\Spec(T^a\times_K A^b)$,
where $T^a\times_K A^b$ denotes the Segre product.  (Recall that this
is the subalgebra of $T^a\otimes_K A^b$ generated by elements
$r\otimes s$ with $r$ and $s$ homogeneous of the same degree.)
Griffith showed in \cite[Theorem 4.5]{Gr} that this variety has the following
properties:
\begin{lemma}\label{Gr} \mbox{}
\begin{enumerate}
\item The variety $X_{a,b}$ is $S_a$, but not $S_{a+1}$.
\item The non-smooth locus of $X_{a,b}$ is the singleton point
  corresponding to the irrelevant maximal ideal.
\end{enumerate}
\end{lemma}

Let $\rho \in W$ be a weakly increasing function as in
Section~\ref{defns}, so in particular $\rho \ge \rho_2$.  We define
the $n^{th}$ inclination of $\rho$ by
\begin{equation*}
t_n \rho (k) = 
\left\{
\begin{array}{cc}
\rho (k) & k \le n \\
\rho(n) + k - n & k > n.
\end{array}
\right.
\end{equation*}
Note that $t_n\rho$ is the maximum element of $W$ that agrees with
$\rho$ on $[0,n]$.  It is trivial that $t_m \rho \ge t_n \rho$
whenever $m \le n$.

Fix $\rho\in W$ and $n\ge 3$, and assume that $\rho|_{[0,n]}\ne \id_{[0,n]}$.
This data determines a (nonempty) increasing sequence $(d_1, d_2,...,
d_s)$, consisting of those indices $m \le n$
satisfying $t_n\rho (m + 1) > t_n\rho(m) =t_n\rho(m -1)$.  
Set $e_i = d_i - \rho(d_i)$ and $r_i = \rho(d_i)$.  
We define a variety $\overline{X}_{\rho, n}$ by
\begin{equation*}
\overline{X}_{\rho, n} = \prod_{i = 1}^s X_{r_i, e_i}.
\end{equation*}
Let $x_{r_i, e_i}$ be the closed point of $X_{r_i, e_i}$ corresponding
to the irrelevant ideal of $T^{r_i} \times_K A^{e_i}$.  Note that
$x_{r_i, e_i}$ has codimension $d_i$.  Define a subvariety $X'_{\rho,
  n} \subset \overline{X}_{\rho, n}$ by
\begin{equation*}
X'_{\rho, n} = \bigcup_{1 \le i \le j \le s} (\{x_{r_i, e_i}\} \times \{x_{r_j, e_j}\} \times 
\prod_{\substack{1 \le \ell \le s \\ \ell \ne i, j}} X_{r_\ell, e_\ell}).
\end{equation*}
Finally, we write $X_{\rho, n} = \overline{X}_{\rho, n} \setminus
X'_{\rho, n}$.

\begin{proposition}
The variety $X_{\rho, n}$ satisfies the generalized Serre condition $S_{\rho}$. 
Moreover, if $\rho'\in W$ is a function such that $t_n \rho' > t_n \rho$, 
then $X_{\rho, n}$ does not satisfy $S_{\rho'}$.
\end{proposition}
\begin{proof}
Let $Y_i$ denote the subvariety
$\{{x}_{r_i, e_i}\} \times_{X_{r_i, e_i}} X_{\rho, n}$.  By assumption,
$Y_i \cap Y_j = \varnothing$.  Note that $Y_i\cong\prod_{j\ne i}
X_{r_j,e_j}\setminus \{x_{r_j,e_j}\}$, so, by Lemma~\ref{Gr}, it is smooth.

First, we calculate the depth at any point $x \in X_{\rho, n}$.  Since
$\cap_i Y_i^c$ is contained in the smooth locus,
$\depth_x(\Osh_{X_{\rho, n}})=\codim(x)$ if $x\notin \cup_i Y_i$.
Now, suppose that $x\in Y_i$.  Let $R = \Osh_{X_{r_i, e_i}, x_{r_i,
    e_i}}$ and $S = \Osh_{X_{\rho, n},x}$, with corresponding maximal
ideals $\fm$ and $\fn$.  Since the projection map $X_{\rho,n}\to
X_{r_i, e_i}$ is flat, $\depth_\fn (S) = \depth_\fm (R) + \depth_\fn
(S/\fm S)$.  There is a similar equation for codimension.  Since the
fiber over $x_{r_i, e_i}$ of this projection is isomorphic to $Y_i$,
$S / \fm S \cong \Osh_{Y_i,x}$.  An application of Lemma~\ref{Gr}
gives $\depth_\fm(R) = r_i$ and $\depth_\fn (S / \fm S) = \codim_{Y_i}
(x)$.  We deduce that
\begin{equation}
\begin{aligned}\label{codimdepth}
\codim (x) & = d_i + \codim_{Y_i}(x) \\
\depth_x(\Osh_{X_{\rho, n},x}) &= r_i + \codim_{Y_i}(x).
\end{aligned}
\end{equation}
whenever $x\in Y_i$.  (Since $r_i<d_i$, it follows that the
Cohen-Macaulay locus (and the smooth locus) is precisely $\cap_i Y_i^c$.)

To show that $X_{\rho, n}$ is $S_{t_n \rho}$, we need only consider
$x\in Y_i$.  Equation~\eqref{codimdepth} shows that the generalized
Serre condition corresponding to $t_{d_i} \rho_{r_i}$ is satisfied at
such an $x$.  The function $t_{d_i} \rho_{r_i}$ is strictly increasing
until $k = r_i$, nonincreasing on the interval $[r_i,d_i]$, and then
strictly increasing afterwards.  It is easily checked that $t_{d_i}
\rho_{r_i} \ge t_n \rho$.

Finally, we suppose that $\rho' \in W$ is a function such that $t_n
\rho' > t_n \rho$.  In particular, there exists a smallest integer $k$
such that $2 < k \le d_s$ and $\rho'(k) > \rho (k)$.  Note that one
cannot have $k\le d_1$ unless $\rho(k)=\rho(d_1)$.  It is clear that
$X_{\rho, n}$ can not be $S_{\rho'}$ if $k \le d_i$ and $\rho (k) =
\rho(d_i)$, since the generic point of $Y_i$ is a codimension $d_i$
point that has depth $\rho (d_i)$.  Suppose now that $d_{i} < k <
d_{i+1}$ and $\rho(k) < \rho(d_{i+1})$.  Then, $\rho$ is strictly
increasing on the interval $[d_{i}, k]$.  It follows that $\rho'(k) >
r_{i} + k - d_i$.  Choose any point $x\in Y_i$ with codimension $k$ in
$X_{\rho,n}$.  By equation \eqref{codimdepth},
$\depth_x(\Osh_{X_{\rho, n},x})=r_i+(k-d_i)$.  Therefore, $X_{\rho,
  n}$ does not satisfy the condition $S_{\rho'}$, since $r_i + (k-d_i)
< \rho'(k)$.

\end{proof}

\section{$S_\rho$-extension and finite $S_\rho$-ification}\label{section4}

In this section, we investigate the ``$S_\rho$-extension problem'' for
any $\rho\in W$ and its relationship to $S_2$-extension.  In
particular, we apply our results to the finite $S_\rho$-ification
problem.

\subsection{$S_\rho$-extension}

\begin{definition}
  Let $U \subset X$ be an open dense subscheme with complement $Z$.  A
  finite morphism $f : Y \to X$ is $S_\rho$ relative to $U$ if
  $f_*\cO_Y$ satisfies the depth conditions in \eqref{prhoformula}.
\end{definition}
If $\ccodim(Z)\ge 1$, this is equivalent to the statement
$f_*\cO_Y\in{}^{p^+_\rho}\cD(X)^{\ge 0}$.  If $\ccodim Z=1$, this
simply means that $f_* \Osh_Y$ is $S_\rho$.

The initial data for $S_\rho$-extension consists of an open dense
subscheme $U$ of the scheme $X$ and a finite dominant morphism $\z_1:
\tU \to U$ that maps generic points to generic points and satisfies
$\z_{1 *}\cO_{\tU}\in{}^{p^+_\rho}\cD(U)^{\ge 0}$.  We will let
$j:U\to X$ denote the inclusion.

\begin{definition}  We say that a scheme $\tX$ together with a
  morphism $\z:\tX\to X$ is an \emph{$S_\rho$-extension} of $(X,\z_1)$
  if it satisfies the following conditions:
\begin{enumerate}
\item \label{ext1}$\tX$ contains $\tU$ as an open dense subscheme;
\item \label{ext2} $\z$ extends $\z_1$ and is finite;
\item \label{ext3} $\z$ is $S_\rho$ relative to $U$; and
\item \label{ext4} $(\tX,\z)$ is universal among finite dominant
  morphisms $f:Y\to X$ which send generic points to
  generic points, 
  which are $S_\rho$ relative to $U$, and whose restriction
  $f|_{f^{-1}(U)}$ factors through $\z_1$.  (In other words, there
  exists a unique $g:Y\to\tX$ such that
  $f=\z\circ\g$.)
\end{enumerate}
We say that $(\tX,\z)$ is a  \emph{weak $S_\rho$-extension} if it
satisfies conditions~\eqref{ext1} and \eqref{ext2} and $\tX$ is
$S_\rho$ off of $\tU$.
\end{definition}

Note that $\z$ is automatically dominant and takes generic point to
generic points.

\begin{rmk}\label{fixed3.5}  Let $f:Y\to X$ be a finite dominant map extending $\z_1$
  and taking generic points to generic points.  Under these
  conditions, $\ccodim_Y (Y-f^{-1}(U))=\ccodim_X Z$ and
  $\codim(y)=\codim(f(y))$.  (Note that we are using the
  equidimensionality of $X$.)  This means that the perversities
  corresponding to $\rho$ on $X$ and $Y$ are related by
  $p^Y_\rho=p^X_\rho\circ f$, so there is no ambiguity in denoting
  both simply by $p_\rho$.  Moreover, the argument given in
  \cite[Proposition 3.5]{AS2} shows that
  $f_*\cO_Y\in{}^{p_\rho^+}\cD(X)^{\ge 0}$ implies that
  $\cO_Y\in{}^{p_\rho^+}\cD(Y)^{\ge 0}$ in the case $\ccodim Z\ge 2$.
  In particular, if $f$ is $S_\rho$-relative to $U$, then $Y$ is
  $S_\rho$ at all points $y$ with $\codim(y)\ge\ccodim Z$.
\end{rmk}

If $\ccodim Z\ge 2$, $(X,\z_1)$ has an $S_2$-extension if and only if
$\cIC^s(X,\z_{1 *} \Osh_{\tU})$ is defined; moreover, it is given by
$\bSpec(\cIC^{s}(X,\z_{1 *} \Osh_{\tU}))\to X$ ~\cite[Theorem
4.5]{AS2}.  (By Remark~\ref{jstar}, $\cIC^s (X, \z_{1 *}
\Osh_{\tU})=j_* \z_{1 *} \Osh_{\tU})$ is automatically a sheaf of
$\cO_X$-algebras if it is defined Thus, the global Spec makes sense.)
We now give the corresponding result for $S_\rho$-extension.

\begin{theorem}\label{extension} Suppose that $\ccodim Z\ge
  2$.  Then, the pair $(X,\z_1)$ has an $S_\rho$-extension if and only
  if $\cIC^{p_\rho}(X,\z_{1 *} \Osh_{\tU})$ is defined and is a sheaf.
  If it exists, it is given by $\bSpec(\cIC^{p_\rho}(X,\z_{1 *}
  \Osh_{\tU}))\to X$ and coincides with the $S_2$-extension.
\end{theorem}

\begin{proof} 
  First, suppose that $(\tX,\z)$ is an $S_\rho$-extension.  By
  property~\eqref{ext3}, $\z_*\Osh_{\tX}\in\cM^{p_\rho}_\pm(X)$;
  restricting to $U$, we see that $\z_{1
    *}\Osh_{\tU}\in\cM^{p_\rho}_\pm(U)$.  This means that
  $\cIC^{p_\rho}(X,\z_{1 *} \Osh_{\tU})$ is defined.  By \cite[Lemma
  3.4]{AS2}, if a coherent sheaf on $X$ extending $\z_{1 *}
  \Osh_{\tU}$ is contained in ${}^{p_\rho^+}\cD(X)^{\ge 0}$, then it
  is isomorphic to $\cIC^{p_\rho}(X,\z_{1 *} \Osh_{\tU})$.  Thus,
  $\cIC^{p_\rho}(X,\z_{1 *} \Osh_{\tU})\cong \z_*\Osh_{\tX}$ is a
  sheaf.

  Conversely, suppose $\cIC^{p_\rho}(X,\z_{1 *} \Osh_{\tU})$ is
  defined and is a sheaf.  Since $p_\rho\ge s$, $\cIC^{s}(X,\z_{1 *}
  \Osh_{\tU})$ is defined, and the same argument given in the previous
  paragraph implies that $\cIC^{p_\rho}(X,\z_{1 *}
  \Osh_{\tU})=\cIC^{s}(X,\z_{1 *} \Osh_{\tU})$.  Accordingly,
  $\tX=\bSpec(\cIC^{p_\rho}(X,\z_{1 *} \Osh_{\tU}))$ is the
  $S_2$-extension of $(X,\z_1)$ and a fortiori satisfies the
  $S_\rho$-extension universal property.  Finally, $\z$ is $S_\rho$
  relative to $U$ because $\z_*\Osh_{\tX}=\cIC^{p_\rho}(X,\z_{1 *}
  \Osh_{\tU})\in {}^{p_\rho^+}\cD(X)^{\ge 0}$. 
\end{proof}

We will see later that under the conditions of the theorem, a weak
$S_\rho$-extension is automatically an $S_\rho$-extension.  We first
prove this for $S_2$-extension.

\begin{proposition}\label{411}
  Suppose that $\ccodim Z\ge 2$ and that $(X,\z_1)$ has an
  $S_2$-extension.  Then a weak $S_2$-extension $(\hX,\hz)$ coincides
  with the $S_2$-extension.
\end{proposition}

\begin{proof} Let $\hat{j}:\tU\to\hX$ be the inclusion.  By
  Remark~\ref{fixed3.5}, $\tU$ is $S_2$ at all points with codimension
  at least $\ccodim Z$, so $\cIC^s(\hX,\cO_{\tU})$ is defined.  Since
  $\hX$ is $S_2$ off of $\tU$, we obtain
  $\cO_{\hX}\cong\cIC^s(\hX,\cO_{\tU})\cong \hat{j}_*\cO_{\tU}$, where
  the last isomorphism holds by Remark~\ref{jstar}.  Applying
  $\hz_*$ gives $\hr_*\cO_{\hX}\cong j_*\z_{1
    *}\cO_{\tU}\cong\cIC^s(X,\z_{1 *}\cO_{\tU})$.  Thus, $\hX$ is
  isomorphic to the $S_2$ extension $\bSpec(\cIC^s(X,\z_{1
    *}\cO_{\tU}))$.
\end{proof}

Let $f : Y \to X$ be a finite map, and let $j' :f^{-1} (U)
\hookrightarrow Y$ be the inclusion.  Write $f_1 = f |_{f^{-1}(U)} :
f^{-1} (U) \to U$.  We say that $Y$ is the integral closure of $X$
relative to $f^{-1}(U)$ (resp. $Y$ is integrally closed relative to
$f^{-1}(U)$) if $Y = \bSpec(\cF)$, where $\cF$ is the integral closure
of $\Osh_X$ in $j_* f_{1 *} (\Osh_{f^{-1}(U)})$ (resp. $\Osh_Y$ is
integrally closed in $j'_* \Osh_{f^{-1}(U)}$).  (See~\cite[Proposition
6.3.4]{Groth}).  If we relax the condition that $\ccodim Z\ge 2$, we
can prove a weaker version of $S_2$-extension as long as the integral
closure of $X$ relative to $\tU$ is finite over $X$.

\begin{proposition}\label{univprop}  Suppose that the integral closure of $X$ relative
  to $\tU$ is finite over $X$.  Let $\hX$ be the associated reduced
  scheme.  The natural morphism $\z : \hX \to X$ is universal with
  respect to finite morphisms $f : Y \to X$ satisfying the following
  properties:
\begin{enumerate}
\item\label{one} $f^{-1} (U)$ is dense in $Y$;
\item\label{two} $f|_{f^{-1}(U)}$ factors through $\z_1$; and
\item\label{three} $Y$ is reduced and integrally closed relative to $f^{-1} (U)$.
\end{enumerate}
\end{proposition}

\begin{rmk}  In this proposition, we do not need to assume that $\z_1$
  is dominant or that it takes generic points to generic points.
\end{rmk}

\begin{rmk} We note that condition~\eqref{three} is equivalent to the following by
Serre's criterion \cite[Theorem 11.5]{Eis}.
\begin{enumerate}
\item[(3')] $Y$ is reduced, and $Y$ satisfies $S_2$ and $R_1$ away
  from $f^{-1} (U)$.
\end{enumerate} 
\end{rmk}

\begin{rmk}If the base $S$ is a Nagata scheme, then the condition that
$\tX$ is finite over $X$ is automatically satisfied.
\end{rmk}

The pair $(\hX,\z)$ given in the proposition is a particular weak
$S_2$-extension of $(X,\z_1)$.  If $\ccodim Z\ge 2$, then it is the
$S_2$-extension by Proposition~\ref{411}.

\begin{proof}
There is a natural morphism of quasi-coherent sheaves of $\Osh_X$-algebras
 \begin{equation}\label{qcoh}
 j_* \z_{1 *} (\Osh_{\tU}) \to j_* f_{1 *} (\Osh_{f^{-1} (U)})
 \end{equation}
defined as follows.  Since $f_1$ factors finitely through $\z_1$, we may write
$f_1 = \z_1 \circ f_1'$ where $f_1' : f^{-1} (U) \to \tU$ is a finite map.  To obtain 
\eqref{qcoh}, we simply apply the functor $j_* \z_{1 *}$ to the 
adjunction map $\Osh_{\tU} \to (f_1')_* (f_1')^* (\Osh_{\tU}) \cong (f_1')_* (\Osh_{f^{-1}(U)}) $.  

By assumption, $f_* \Osh_Y$ is integrally closed in $j_* (f_{1 *})
(\Osh_{f^{-1} (U)})$.  Moreover, the map  $f_* \Osh_Y\to j_* (f_{1 *})
(\Osh_{f^{-1} (U)})$ is injective;  indeed, since $f^{-1}(U)$ is dense in $Y$ and
$Y$ is reduced, the morphism $\Osh_Y \to (j')_* \Osh_{f^{-1} (U)}$ is
injective.   Finally, \cite[6.3.5]{Groth} shows that
there is a canonical morphism $Y \to \hX$.

\end{proof}

\begin{lemma}\label{R1}
  Suppose that $f : Y \to X$ is a finite dominant morphism of reduced
  Noetherian schemes that takes generic points to generic points.
  Suppose that $U$ is a dense open subscheme of $X$ such that $f$
  induces an isomorphism between $f^{-1} (U)$ and $U$.  Let $x\in
  X\setminus U$ be a codimension one point that is regular.  Then,
  there exists a unique point $y \in Y$ (necessarily of codimension
  one) lying above $x$.  Moreover, $y$ is regular, and the map $f^*_y
  : \Osh_{X,x} \to \Osh_{Y,y}$ is an isomorphism.
\end{lemma}

\begin{proof}
  The hypotheses imply that $f$ induces a bijection between the
  irreducible components of $Y$ and $X$.  Since $x$ is regular, it
  lies in a single component of $X$.  This means that $f^{-1}(y)$ can
  only intersect the corresponding component of $Y$.  We may
  accordingly assume without loss of generality that $X$ and $Y$ are
  irreducible, hence integral, and $f$ induces an isomorphism of
  function fields $K(X)\cong K(Y)$.

Since $x$ is regular, $f^{-1}(x)$ contains a single point $y$ by the
valuative criterion of properness.  Moreover, the local ring $\Osh_{Y,y}$ 
dominates the valuation ring $f^*_y\Osh_{X,x}$ in $K(Y)$, so they are
equal.  Hence, $\Osh_{Y,y}$ is also a discrete valuation ring, and
$f^*_y : \Osh_{X,x} \to \Osh_{Y,y}$ is an isomorphism.

\end{proof}

\begin{proposition}\label{iso}   Let $X$, $Y$, and $f:Y\to X$ satisfy
  the conditions in Proposition~\ref{univprop}.  Suppose that $f$ is
  dominant and takes generic points to generic points and that the
  given factorization map $g_1:f^{-1}(U)\to\tU$ is an isomorphism.  If
  we further assume that $Y$ (and hence $\tU$) is $S_2$, then the map
  $g: Y \to \hX$ constructed in Proposition~\ref{univprop} is an
  isomorphism.
\end{proposition}

\begin{proof} Suppose that we can find an open subset $\tU\subset
  V\subset \hX$ such that $g|_{g^{-1}(V)}:g^{-1}(V)\to V$ is an
  isomorphism and $\ccodim (\hX\setminus V)\ge 2$.  Since $g$ is a
  weak $S_2$-extension of $g|_{g^{-1}(V)}:g^{-1}(V)\to V$, by
  Proposition~\ref{411}, $g$ is the $S_2$-extension.  Moreover, the
  identity map $V\to V$ has an $S_2$-extension, namely, $\hX\to \hX$.
  It follows that $g$ is an isomorphism.

  Let $V$ be the complement of the support of the coherent
  $\cO_{\hX}$-module $f_*\cO_Y/\cO_{\hX}$. This is an open set
  containing $\tU$.  Note that $g|_{g^{-1}(V)}:g^{-1}(V)\to V$ is a
  continuous bijection; finiteness implies that it is closed, hence a
  homeomorphism.  Since the induced map of sheaves is obviously an
  isomorphism, we obtain a scheme isomorphism. By Lemma~\ref{R1}, $V$
  contains all codimension one points not in $\tU$.  It follows that
  $\codim \tX\setminus
  V\ge 2$.  The same holds for the c-codimension.  Indeed, the
  equidimensionality hypothesis implies that a point $x$ has the
  same codimension in any irreducible component containing
  it.
\end{proof}

Putting together Propositions~\ref{411} and \ref{iso} and using the
fact that a weak $S_\rho$-extension is a weak $S_2$-extension, we obtain the
following result above weak $S_\rho$-extensions.

\begin{theorem} \label{weak} \mbox{}
\begin{enumerate}  \item Suppose $\ccodim Z\ge 2$ and $\cIC^s (X,
  \z_{1 *} \Osh_{\tU})$ is defined.  Then any weak $S_\rho$-extension of
  $(X,\z_1)$ coincides with the $S_2$-extension.
\item Suppose that $\tU$ is reduced and the integral closure of $X$
  relative to $\tU$ is finite over $X$.  Then any reduced weak
  $S_\rho$-extension $(Y,f)$ of $(X,\z_1)$ that is $R_1$ outside of
  $f^{-1}(U)$ coincides with the weak $S_2$-extension $(\hX,\z)$
  constructed in Proposition~\ref{univprop}.
\end{enumerate}
\end{theorem}

\subsection{Finite $S_\rho$-ification}

We now apply our results to the finite $S_\rho$-ification problem.
Recall that a finite $S_\rho$-ification of a scheme $X$ is an $S_\rho$
scheme $\hX$ together with a finite birational map $\hz:\hX\to X$.  If
we let $U$ be an open dense subset of $X$ on which $\hz$ is an
isomorphism, we see that a finite $S_\rho$-ification may be viewed as
a weak $S_\rho$-extension of the identity map $U\to U$.  (Observe,
however, that a weak $S_\rho$-extension of the identity can be
defined without assuming $U$ is $S_\rho$.)

\begin{theorem}\label{srification}\mbox{}
\begin{enumerate} \item \label{srho1} Assume that the $S_\rho$ locus
  of $X$ contains an open dense set whose complement has c-codimension
  at least $2$.  Then,
\begin{enumerate}
\item \label{parta} If the complex $\cIC^{p_\rho} (X, \Osh_U)$ is a sheaf for such
  an open set $U$, then $\cIC^{p_\rho} (X, \Osh_V)$ is a sheaf for any
  such $V$, and they are all isomorphic.
\item \label{partb} The scheme $X$ has a finite $S_\rho$-ification
  which is an isomorphism off of a closed set of c-codimension at
  least $2$ if and only if $\cIC^{p_\rho} (X, \Osh_U)$ is a sheaf for
  any such open set $U$.
\item If it exists, it is unique and coincides with the unique finite
  $S_2$-ification which is an isomorphism on the $S_2$ locus $W$.  In
  particular, the $S_2$ and $S_\rho$ loci coincide, and this finite
  $S_\rho$-ification can be given explicitly as $\bSpec(\cIC^{s} (X,
  \Osh_W))\to X$.
\end{enumerate}

\item \label{srho2} Assume that the $S_\rho$ locus of $X$ contains a reduced
  open dense set $U$ such that the integral closure $\hX$ of
  $X$ relative to $U$ is finite over $X$.  Then, $X$ has a finite
  $S_\rho$-ification which is $R_1$ off of $U$ if and only if $\hX$ is
  $S_\rho$.  If such a finite $S_\rho$-ification exist, it coincides
  with the unique $S_2$-ification which is $R_1$ off of $U$.
\end{enumerate}
\end{theorem}
\begin{proof}
  The second part follows immediately from Proposition~\ref{univprop}
  and Theorem~\ref{weak}.  Note that since $U$ is reduced, the
  integral closure of $X$ relative to $U$ is automatically reduced.
  Thus, it is unnecessary to pass to the associated reduced scheme as
  in Proposition~\ref{univprop}.

For the first part, assume that $U\hookrightarrow X$ is open, dense,
and $S_\rho$ with $\ccodim (X\setminus U)\ge 2$.  Since a finite
$S_\rho$-ification that is an isomorphism over $U$ is the same thing
as a weak $S_\rho$-extension of $\id:U\to U$, Theorem~\ref{extension}
implies that this exists if and only if $\cIC^{p_\rho} (X, \Osh_U)$ is
a sheaf, in which case it coincides with $\cIC^{s} (X, \Osh_U)\cong
j_{*}(\cO_U)$ (where the last isomorphism uses Remark~\ref{jstar}).

Suppose that this is the case and that
$V\overset{i}{\hookrightarrow}X$ is another open, dense, $S_\rho$
subscheme with the c-codimension of its complement at least $2$.  Both
$j_*(\cO_U)$ and $i_*(\cO_V)\cong\cIC^{s} (X, \Osh_V)$ are $S_2$
coherent sheaves extending $\cO_{U\cap V}$, so they are isomorphic;
they are both isomorphic to $\cIC^{s} (X, \Osh_{U\cap V})$.  (Since
$\ccodim (X\setminus(U\cap V))\ge 2$, the $\IC$ sheaf is defined.)  In
particular, $i_*(\cO_V)\in {}^{p_\rho}\cD(X)^{\ge 0}$, so it equals
$\cIC^{p_\rho} (X, \Osh_V)$.  Thus, this complex is a sheaf and
coincides with $\cIC^{p_\rho} (X, \Osh_U)$.

Part~\eqref{partb} now follows from Theorems~\ref{extension} and
\ref{weak} because such a finite $S_\rho$-ification is the same thing
as a weak $S_\rho$ extension of $\id:U\to U$ (with $U$ as above).
Part~\eqref{parta} also shows that if this $S_\rho$-ification exists,
it is unique.  Finally, observe that $U\subset W$, so $\cIC^{s} (X,
\Osh_{U})=\cIC^{s} (X, \Osh_{W})$.  This implies that $\bSpec(\cIC^{s}
(X, \Osh_{U})|_W)|\cong W$, so $W$ is $S_\rho$.

\end{proof}

\begin{rmk}  Even in the context of part~\eqref{srho1} of the
  theorem, there can be other finite $S_\rho$-ifications which are not
  isomorphisms over an open, dense set whose complement has
  c-codimension at least $2$.  Indeed, take any variety which is
  $S_2$, but not $R_1$.  Then, the identity map and the normalization
  are non-isomorphic finite $S_2$-ifications.
\end{rmk}

\begin{corollary} If the non-$S_2$ locus of $X$ has c-codimension at
  least $2$, then there is a unique finite $S_2$-ification which is an
  isomorphism on a dense, open set whose complement has c-codimension
  at least $2$.  Moreover, it is an isomorphism over the $S_2$ locus.
 \end{corollary}
 \begin{proof} This follows from the theorem because $\cIC^s$ takes
   sheaves to sheaves.
\end{proof}


\begin{thebibliography}{99}
\bibitem{AS2} P. Achar and D.~S. Sage, ``Perverse coherent sheaves and
  the geometry of special pieces in the unipotent variety,'' Adv.
  Math. {\bf 220} (2009), 1265--1296.

\bibitem{AB} D. Arinkin and  R. Bezrukavnikov, ``Perverse coherent
  sheaves,'' Mosc. Math. J. {\bf 10} (2010), 3--29,
 
     

\bibitem{Eis} D. Eisenbud, {\em Commutative algebra with a view toward
    algebraic geometry}, Graduate Texts in Mathematics {\bf 150},
  Springer-Verlag, New York, 1995.

\bibitem{faltings}
G. Faltings, ``\"Uber Macaulayfizierung,'' Math. Ann. {\bf 238} (1978),
175--192.

\bibitem{Gr} P. Griffith, ``Induced formal deformations and the
  Cohen-Macaulay property,'' Trans. Amer. Math. Soc. {\bf 353} (2000), 77-93.

\bibitem{Groth} A. Grothendieck, ``\'El\'ements de g\'eom\'etrie
  alg\'ebrique (r\'edig\'es avec la collaboration de Jean
  Dieudonn\'e): II. \'Etude globale \'el\'ementaire de quelques
  classes de morphismes,'' Pub. Math. IH\'ES {\bf 8} (1961), 5-222.



\bibitem{kawasaki} T. Kawasaki, ``On Macaulayfication of Noetherian
  schemes,'' Trans. Amer.  Math. Soc. {\bf 352} (2000), 2517--2552.

\bibitem{schenzel} P. Schenzel, ``On birational Macaulayfications and
  canonical Cohen--Macaulay modules,'' J. Algebra {\bf 275} (2004),
  751--770.

\end{thebibliography}
\end{document}